\documentclass[graybox]{svmult}

\usepackage{helvet}         
\usepackage{courier}        
\usepackage{type1cm}        
\usepackage{makeidx}         
\usepackage{graphicx}        
\usepackage{multicol}        
\usepackage[bottom]{footmisc}

\usepackage{amsmath}
\usepackage{amssymb}
\usepackage{amsfonts}
\newcommand{\CN}{\mathbb{C}}
\newcommand{\ZN}{\mathbb{Z}}
\newcommand{\mc}{\mathcal}
\newcommand{\abs}[1]{\lvert #1\rvert}
\newcommand{\ovl}{\overline}
\renewcommand{\th}{\vartheta}

\begin{document}

\title*{A different approach to positive traces on generalized $q$-Weyl algebras}
\author{Daniil Klyuev}
\institute{Daniil Klyuev \at Department of Mathematics, Massachusetts Institute of Technology, Cambridge, MA 02139, USA \email{klyuev@mit.edu}}
%
%
\maketitle

\abstract{Positive twisted traces are mathematical objects that could be useful in computing certain parameters of superconformal field theories. The case when $\mc{A}$ is a $q$-Weyl algebra and $\rho$ is a certain antilinear automorphism of $\mc{A}$ was considered in~\cite{K}. Here we consider more general choices of $\rho$. In particular, we show that for $\rho$ corresponding to a standard Schur index of a four-dimensional gauge theory a positive trace is unique.
}
\section{Introduction}
\label{SecIntroduction}
Let $\mc{A}$ be a generalized $q$-Weyl algebra, it is generated by $u,v,Z,Z^{-1}$ with relations $ZuZ^{-1}=q^2u$, $ZvZ^{-1}=q^{-2}v$, $uv=P(q^{-1}Z)$, $vu=P(qZ)$, where $P$ is a Laurent polynomial. Let $\rho$ be a conjugation on $\mc{A}$. We are interested in positive traces on $\mc{A}$. We assume that $0<q<1$.

Positive traces are defined as follows. Let $\mc{A}$ be a noncommutative algebra over $\CN$, $\rho$ be an antilinear automorphism of $\mc{A}$. Let $g=\rho^2$. We say that a linear map $T\colon \mc{A}\to\CN$ is a {\it $g$-twisted trace} if $T(ab)=T(bg(a))$ for all $a,b\in\mc{A}$. A $g$-twisted trace is {\it positive} if $T(a\rho(a))>0$ for all nonzero $a\in\mc{A}$. 



Positive traces for $q$-Weyl algebras appear in the study of the Coulomb branch of 4-dimensional superconformal field theories~\cite{DG, GT}.

We will consider $g$-twisted traces for automorphisms $g$ such that $g(Z)=Z$. In this case we have $T(ZaZ^{-1})=T(a)$ for all $a\in \mc{A}$. This means that $T$ is zero for any expression $u^kR(z)$ and $v^kR(z)$, where $R$ is a Laurent polynomial and $k$ is a positive integer. Hence $T$ is uniquely defined by its values on Laurent polynomials.

More precisely, we will consider $g=g_k$ such that $g(Z)=Z$, $g(u)=q^{k}Z^{-k} u$, $g(v)=q^{k}Z^{k} v$ for some integer $k$.

We will identify a trace $T$ on $\mc{A}$ with its restriction to $\CN[Z,Z^{-1}]\subset \mc{A}$ below.

\section{Traces via formal integral}
\label{SecMain}
Let $P$ be a Laurent polynomial that has $n$ nonzero roots. For simplicity, we will consider the case when $P$ has no roots with absolute value $q^{\pm 1}$, but we expect that our methods could be modified to work in the general case.

Similarly to~\cite{K}, $T\colon \CN[Z,Z^{-1}]\to \CN$ is a $g_k$-twisted trace if and only if $T(uvR(q^{-1}Z)-vR(Z)q^kZ^{-k}u)=0$ for all $R\in\CN[z,z^{-1}]$. We have \[vR(Z)q^kZ^{-k}u=vuR(qZ)Z^{-k}=P(qZ)R(qZ)Z^{-k}.\] Hence $T$ is a twisted trace if and only if
\[T(P(q^{-1}Z)R(q^{-1}Z)-P(qZ)R(qZ)Z^{-k})=0\] for all Laurent polynomials $R$.
\subsection{Properties of the vector space of formal power series.}
Let $V=\CN[[z,z^{-1}]]$ be the linear space of two-sided formal power series $\sum_{i=-\infty}^{\infty}a_i z^i$. Note that $V$ is a $\CN[z]$-module. Let $CT$ be the constant term map $\sum a_iz_i\mapsto a_0$.

For a nonzero Laurent polynomial $P$ let $P_r^{-1}\in \CN((z))$ be its inverse in Laurent series and $P_l^{-1}\in \CN((z^{-1}))$ be its inverse in Laurent series in the opposite direction.

\begin{lemma}
\label{LemMultiplicationByP}
The map $w(z)\mapsto w(z)P(z)$ is surjective and has kernel of dimension $n$. Moreover, if $P$ has distinct roots $\alpha_1,\ldots,\alpha_k$ with multiplicity $m_1,\ldots,m_k$ respectively, then the kernel is linearly spanned by the elements $f_{i,j}=\sum_{l\in\ZN} l(l-1)\cdots (l-i)\alpha_j^{-l}z^l$, where $j=1,\ldots,k$, $i=0,\ldots,m_j-1$.
\end{lemma}
\begin{proof}
Let $w=\sum a_iz^i$ be an element of $V$. Let $w_+=\sum_{i\geq 0}a_iz^i$, $w_-=\sum_{i<0} a_iz^i$. Then $w_+=(w_+P_r^{-1})P$ and $w_-=(w_-P_l^{-1})P$, so that $w=(w_+P_r^{-1}+w_-P_l^{-1})P$. This proves surjectivity.

We turn to the statement on the dimension of the kernel. Using the surjectivity and induction on $n$ it is enough to consider the case $P=z-a$ for $a\in \CN^{\times}$. In this case direct computation gives that the kernel is one-dimensional and spanned by $\sum a^{-i}z^i$.

It is not hard to see that the $n$ elements in the statement of the theorem belong to the kernel and are linearly independent, hence they span the kernel.
\end{proof}

\subsection{Growth condition on traces.}
Let $T\colon \CN[z,z^{-1}]\to\CN$ be a trace. Let $T(z^i)=c_i$. Define an element $w$ of $\CN[[z,z^{-1}]]$ by $w=\sum c_i z^{-i}$. Then $T(R(z))=CT(R(z)w(z))$ for any Laurent polynomial $R$.

Let $T$ be a positive trace with respect to a conjugation $\rho$ such that $\rho(Z)=Z^{-1}$. Assume that $T(1)=1$, so that $c_0=1$.
\begin{lemma}
\label{LemCiAtMostOne}
For all integers $k$ we have $\abs{T(z^k)}< 1$. Hence $\abs{c_i}<1$ for all $i\neq 0$.
\end{lemma}
\begin{proof}
Consider an element $a=1+\alpha Z^k$. By positivity of $T$ we have $T(a\rho(a))>0$. Note that \[a\rho(a)=(1+\alpha Z^k)(1+\overline{\alpha}Z^{-k})=1+\abs{\alpha}^2+\alpha Z^k+\overline{\alpha}Z^{-k}.\] Hence \[T(a\rho(a))=1+\abs{\alpha}^2+\alpha T(z^k)+\ovl{\alpha}T(z^{-k}).\] Since this expression is positive for all $\alpha$, we get $\abs{T(z^k)}<1$.
\end{proof}

Assume that $\rho^2=g_l$, so that $T$ is a positive $g_l$-twisted trace.
\begin{theorem}
The element $w(z)$ coincides with the Laurent expansion of a function $f(z)$ holomorphic on a neighborhood of the circle $\abs{z}=1$. The function $f(z)P(qz)$ can be analytically continued to a meromorphic function on $\CN^{\times}$ holomorphic on $1<\abs{z}<q^{-2}$.
\end{theorem}
\begin{proof}
For all Laurent polynomials $R$ we have \[T\big(P(qz)R(qz)-z^l P(q^{-1}z)R(q^{-1}z)\big)=0.\] Hence
\[CT\big(P(qz)R(qz)w(z)\big)=CT\big(z^l P(q^{-1}z)R(q^{-1}z)w(z)\big).\]
It follows that
\[CT\big(P(z)R(z)w(q^{-1}z)\big)=CT\big(q^lz^l P(z)R(z)w(qz)\big),\]
hence
\[CT\bigg(R(z)\big(P(z)w(q^{-1}z)-q^lz^l P(z)w(qz)\big)\bigg)=0.\]
Since this holds for any Laurent polynomial $R$, we get
\begin{equation}
\label{EqQuasiperiodicityForW}
P(z)w(q^{-1}z)=q^lz^lP(z)w(qz).
\end{equation}

Hence $w(q^{-1}z)-q^lz^lw(qz)$ belongs to the kernel of multiplication by $P(z)$. Using Lemma~\ref{LemMultiplicationByP}, we get
\[w(q^{-1}z)-q^lz^lw(qz)=\sum_{i,j} r_{i,j}f_{i,j}.\] We note that the coefficients on the left-hand side have growth at most $q^{-N}$ when $N$ tends to $+\infty$. On the right-hand side the coefficients have growth $\abs{\alpha_j}^NN^a$, where $\alpha_j$ is a root with the largest absolute value such that $f_{i,j}\neq 0$ for some $i$ and $a$ is a nonnegative integer. In fact, $a$ is the largest number such that for some $k$ we have $f_{a,k}\neq 0$ and $\abs{\alpha_k}=\abs{\alpha_j}$.  Since there are no roots with absolute value $q^{-1}$, the right-hand side grows as $\alpha^NN^a$, where $\abs{\alpha}<q^{-1}$. Using that the coefficient of $z^N$ in $w(q^{-1}z)-q^lz^lw(qz)$ equals to $q^{-N}c_N+O(q^N)$, we get that $c_N=O(\kappa_+^NN^a)$, where $0<\kappa_+<1$. Similarly, $c_{-N}=\kappa_-^NN^b$, where $0<\kappa_-<1$ and $b$ is a nonnegative integer. This proves the first statement.

We turn to the second statement. Consider the expansion $P(qz)w(z)$ for the function $P(qz)f(z)$ in a neighborhood of the circle $\abs{z}=1$. The coefficients of $P(qz)w(z)$ decay exponentially in the negative direction by the reasoning above. It follows from~\eqref{EqQuasiperiodicityForW} that $P(qz)w(z)=sq^{2l}z^l P(qz)w(q^2z)$. Using Lemma~\ref{LemCiAtMostOne} again, we see that the coefficients of $z^N$ in $P(qz)w(z)$ is $O(q^{2N}\kappa_+^{N}N^a)$ when $N$ tends to $+\infty$. It follows that $P(qz)w(z)$ gives a function holomorphic in the annulus $r_1<\abs{z}<r_2$, where $r_1<1<q^{-2}<r_2$. Hence $f(z)$ can be analytically continued to a meromorphic function in the annulus $r_1<\abs{z}<r_2$. Using~\eqref{EqQuasiperiodicityForW} again, we see that $f(z)$ can be analytically continued to a meromorphic function on $\CN^{\times}$.
\end{proof}
\section{The positivity condition}
Consider the antilinear automorphism $\rho$ such that $\rho(Z)=Z^{-1}$, $\rho(u)=Z^k q^{k}v$, $\rho(v)=Z^{-k} q^{k} u$. Similarly to~\cite{K} $\rho$ is well-defined when $P(z)=\ovl{P}(z^{-1})$. We have $\rho^2=g_{2k}$. Since we have a full classification for the case $k=0$, assume that $k\neq 0$. The case when $k=\frac{n}{2}$ is of particular interest, since it corresponds to the ``standard'' Schur index of a four-dimensionall gauge theory~\cite{G}.

As in~\cite{K}, Lemma 3.2, it is enough to check the positivity condition $T(a\rho(a))>0$ when $a\in \CN[Z,Z^{-1}]$ or $a\in u\CN[Z,Z^{-1}]$. The first condition translates to $T(R(z)\ovl{R}(z^{-1}))>0$ for all nonzero Laurent polynomials $R$, whereas the second one translates to 
\[T(uR(qZ)Z^kq^kv\ovl{R}(qZ^{-1}))>0,\] which is equivalent to \[T(z^k P(q^{-1}z)R(q^{-1}z)\ovl{R}(qz^{-1})>0.\]

Reasoning as in~\cite{K}, Theorem 3.3, this is equivalent to $w(z)$ and $z^k P(z)w(qz)$ being nonnegative on the unit circle $\abs{z}=1$.

From~\eqref{EqQuasiperiodicityForW} we have
\begin{equation}
\label{EqRealQuasiperiodicityForW}
w(q^{-1}z)=q^{2k}z^{2k}w(qz).
\end{equation} Let $\th(z)=\th_{11}(\frac{\log z}{2\pi i})$, where $\th_{11}$ is a Jacobi theta function for period $\tau=\frac{\log q}{\pi i}$. Hence $\th(1)=0$. It is well-known that any meromorphic function satisfying~\eqref{EqRealQuasiperiodicityForW} can be expressed as $w=cz^l\frac{\prod_i \th(\frac{z}{\alpha_i})}{\prod_j \th(\frac{z}{\beta_j})}$, where $\alpha_1,\ldots,\alpha_N$ and $\beta_1,\ldots,\beta_M$ are the zeroes and poles of $w$. Since $w(z)P(qz)$ is holomorphic on $1<\abs{z}<q^{-2}$, elements $q\beta_1,\ldots,q\beta_M$ are roots of $P$. Moreover, in order to cancel out poles,  $q\beta_1,\ldots,q\beta_M$ should satisfy $q<\abs{q\beta_i}<q^{-1}$. Hence we can assume that $M$ equals to the number of zeroes of $P$ in the annulus $q<\abs{z}<q^{-1}$.

We have $\th(q^2z)=q^{-1}z^{-1}\th(z)$. In particular, multiplying $\alpha_1$ or $\beta_1$ by a power of $q^2$, we can assume that $l=0$. Also \[\th(q^2\frac{z}{\alpha})=q^{-1}\alpha z^{-1}\th(\frac{z}{\alpha}).\] It follows that \[w(q^2z)=w(z)\frac{\prod_i (q^{-1}\alpha_i z^{-1})}{\prod_j (q^{-1}\beta_j z^{-1})}=w(z)q^{-N+M}z^{M-N}\frac{\prod_i\alpha_i}{\prod_j\beta_j}.\]
Comparing with~\eqref{EqRealQuasiperiodicityForW} we get $N=M-2k$, $\prod_i\alpha_i=\prod_j\beta_jq^{-N+M}$. Since $P(z)=\ovl{P}(z^{-1})$ the product $\prod_j\beta_j$ has absolute value $q^{-M}$. Multiplying $Z$ by a complex number with absolute value one if necessary, we can assume that $\prod_j\beta_j=q^{-M}$. We get a classification similar to~\cite{K}: possible $w$ are parametrized by $n-2k$ parameters $\alpha_i$ with fixed product $q^N$, and $c$. Note that for $k=\frac{M}{2}$ we have no numerator and the function $w$ is unique up to scaling.

Similarly to~\cite{K}, Section 3.2 the denominator of $w$ is positive on $S^1$ and $qS^1$. Also reasoning similarly to~\cite{K}, Section 3.2, we see that when $c$ is positive and $\alpha_1,\ldots,\alpha_N$ are divided into pairs $\alpha_i$ and $\alpha_j=q^2\ovl{\alpha_i^{-1}}$, the numerator is positive on $S^1$ and $qS^1$. It follows that the cone of positive traces has maximal possible real dimension $N=M-2k$ (one if $N=0$).

Hence we obtain the following
\begin{theorem}
Let $\mc{A}$, $\rho$ be as above. Then the dimension of the cone of positive traces is $N=M-2k$, one if $N=0$, and there are no positive traces if $N<0$. In particular, when $k=\frac{n}{2}$, a positive trace exists if and only if all roots of $P$ belong to the annulus $q<\abs{z}<q^{-1}$, and if it is exists, it is unique up to scaling.
\end{theorem}
In particular, we confirm a conjecture of Gaiotto and Teschner~\cite{GT} in the case of abelian gauge theories: if $A$ is a $K$-theoretic Coulomb branch and $\rho$ corresponds to the ``standard'' Schur index, then the positive trace on $A$, if it exists, is unique up to scaling.
\begin{acknowledgement}
I would like to thank the organizers of the LT-15 conference for the opportunity to listen to interesting talks, give a talk and write this contribution. I am thankful to Davide Gaiotto for explaining to me the physical meaning of different choices of the conjugation $\rho$.
\end{acknowledgement}

\end{document}